\documentclass[11pt,a4paper]{article}

\usepackage[utf8]{inputenc}
\usepackage{authblk}
\usepackage{lineno,hyperref}
\modulolinenumbers[5]

\usepackage{amssymb}
\usepackage{amsmath}
\usepackage{amsthm}
\usepackage{chemarrow}
\usepackage{graphicx}
\usepackage[left=2.5cm,right=2.5cm,top=2.5cm,bottom=2.5cm]{geometry}

\newtheorem{thm}{Theorem}

\newcommand*{\Ln}{\text{Ln}}

\newtheorem*{remark}{\textnormal{\textbf{Remark}}}

\begin{document}

\title{Semistability of complex balanced kinetic systems with arbitrary time delays}

\author[1,2]{György Lipták}
\author[1,2]{Katalin M. Hangos}
\author[3]{Mihály Pituk}
\author[1,4]{Gábor Szederkényi} 

\affil[1]{\small Process Control Research Group, Systems and Control Laboratory, Institute for Computer Science and Control (MTA SZTAKI), Hungarian Academy of Sciences, \mbox{Kende u. 13-17, H-1111 Budapest, Hungary}} 
\affil[2]{\small Department of Electrical Engineering and Information Systems, University of Pannonia, \mbox{Egyetem u. 10, H-8200 Veszprém, Hungary}}
\affil[3]{\small Department of Mathematics, University of Pannonia, \mbox{Egyetem u. 10, H-8200 Veszprém, Hungary}}
\affil[4]{\small Faculty of Information Technology and Bionics, Pázmány Péter Catholic University, \mbox{Práter u. 50/a, H-1083 Budapest, Hungary}}

\date{}

%
%

\maketitle

\begin{abstract}
In this letter we introduce a class of delayed kinetic systems derived from mass action type reaction network models. We define the time delayed positive stoichiometric compatibility classes and the notion of complex balanced time delayed kinetic systems. We prove the uniqueness of equilibrium solutions within the time delayed positive stoichiometric compatibility classes for such models. In our main result we prove the semistability of the equilibrium solutions for complex balanced systems with arbitrary time delays using an appropriate Lyapunov-Krasovskii functional and LaSalle's invariance principle. As a consequence, we obtain that every positive complex balanced equilibrium solution is locally asymptotically stable relative to its positive stoichiometric compatibility class.
\end{abstract}


\section{Introduction}
The class of kinetic systems has proven to be a useful representation of nonnegative system models not only in biochemistry, but also in other areas like population or disease dynamics, process systems, and even transportation networks \cite{Erdi1989,Haddad2010,Walter1999}. A network-based description is often advantageous to describe key properties of potentially large, complex systems with many components \cite{Volpert1972,Osipenko2007}. Kinetic systems are naturally equipped with a network (i.e., directed graph) structure called the reaction graph, which is the abstraction of a set of chemical reactions, where the chemical complexes and reactions can be represented by vertices and directed edges, respectively. One of the primary aims of chemical reaction network theory (CRNT) is to discover relations between the dynamical behaviour and the graph structure of kinetic systems \cite{Horn1972,Feinberg1979,Angeli2009}. Probably the most widely known results of general importance in this field are the Deficiency Zero and Deficiency One Theorems \cite{Feinberg1987} and more recently the notion of absolute concentration robustness \cite{Shinar2010}.

The detailed balance property of a thermodynamic system, defined originally by Boltzmann in the 19th century, means that at equilibrium, each elementary reaction step is equilibriated by the corresponding reverse reaction. A more general condition is complex balance, which requires that the signed sum of incoming and outgoing reaction rates at equilibrium is zero for each complex in a chemical reaction network \cite{Horn1972a,Feinberg1972,Dickenstein2010}. It is worth remarking that complex balance does not depend on a particular equilibrium (if there exist multiple equilibria in a system), but it is a property of a chemical reaction network itself \cite{Horn1972}. For a historical review of the notions of detailed and complex balance, see \cite{Gorban2013}. 
Generally, complex balance is related both to the structure and to the parameters of chemical reaction networks. Firstly, complex balance implies that each component of the reaction graph is strongly connected (i.e., the reaction network is weakly reversible) \cite{Horn1972}. It is also important that deficiency zero weakly reversible reaction networks are complex balanced for any positive values of the reaction rate coefficients \cite{Feinberg1979,Feinberg1987}. However, complex balance becomes a parameter-dependent property when the deficiency of the network is higher than zero. The main significance of complex balance in systems and control theory stands in its stability implications \cite{Son:2001}. According to the Global Attractor Conjecture, complex balanced kinetic systems are globally stable in the positive orthant with a logarithmic Lyapunov function that does not depend on the model parameters. The conjecture was proved for several special cases such as one linkage class networks \cite{Anderson2011}, and a possible proof for the general problem has recently appeared in \cite{Craciun2015}. Using the non-uniqueness of reaction graphs corresponding to kinetic models \cite{Liptak2015}, a feedback design method was proposed in \cite{Liptak2016} that transforms a polynomial model into a complex balanced closed loop system via nonlinear state feedback.

Time-delays are often present in natural and technological processes, and the detailed mathematical treatment of such delays is sometimes necessary to model and understand important observed dynamical phenomena  \cite{Stepan1989,Fridman2014}. An excellent summary of the fundamental results on nonnegative and compartmental systems with time-delay can be found in Chapter 3 of \cite{Haddad2010}, where simple algebraic necessary and sufficient conditions are given for the asymptotic stability of delayed linear nonnegative models including linear compartmental systems. Among other results, the semistability of an important special class of nonlinear compartmental systems for arbitrary time-delays was shown in \cite{Chellaboina2008}.

Motivated by the above results, the purpose of this paper is to introduce the complex balance condition for kinetic systems with delayed reactions, and to study the stability properties of such systems using logarithmic Lyapunov-Krasovskii functionals and LaSalle's invariance principle.

Throughout the paper, we will use the following notations. If $N$ is a positive integer, ${\mathbb{R}}^N$ denotes the $N$-dimensional space of real column vectors. The symbols $\mathbb{R}^N_{+}$ and $\overline{\mathbb{R}}^N_{+}$ denote the set of  (element-wise) positive and nonnegative vectors in~$\mathbb{R}^N$, respectively. For $x,y \in \mathbb{R}^N_{+}$, the vector $\frac{x}{y}\in{\mathbb{R}}^N_{+}$ is defined by ${(\frac{x}{y})}_i=\frac{x_i}{y_i}$ for $i=1,\dots,N$. For $x, y\in\overline{\mathbb{R}}^N_{+}$,  the vector exponential $x^y$ is defined as $x^y = \prod_{i=1}^N x_i^{y_i}$. 
The mapping $\Ln:~ \mathbb{R}^N_{+} \rightarrow \mathbb{R}^N$ is the element-wise logarithmic mapping defined by ${(\Ln(x)})_i = \ln(x_i)$ for $x\in\mathbb{R}^N_{+}$ and $i=1,\dots,N$. Recall that $\ln(x^y) = y^T\Ln(x)$ for $x\in\mathbb{R}^N_{+}$ and $y\in\mathbb{R}^N$, where~$(\cdot)^T$ denotes the transpose, and $\Ln(\frac{x}{y}) = \Ln(x) - \Ln(y)$ whenever $x,y\in\mathbb{R}^N_{+}$.
For every $\tau\geq0$, the symbol $\mathcal{C} = C([-\tau, 0], \mathbb{R}^N)$ denotes the Banach space of continuous functions mapping the interval $\left[-\tau, 0\right]$ into $\mathbb{R}^N$ with the norm $\|\psi\|= \sup_{-\tau\leq s \leq 0}|\psi(s)|$ for $\psi\in\mathcal{C}$, where $|\cdot|$ denotes the Eucledian norm in $\mathbb{R}^N$. Finally, let
$\mathcal{C}_{+} = C([-\tau, 0], \mathbb{R}^N_{+})$ and $\overline{\mathcal{C}}_{+}=C([-\tau, 0], \overline{\mathbb{R}}^N_{+})$ denote the set of  positive and nonnegative functions in~$\mathcal{C}$.

\section{Kinetic systems with time delays}
In this section, we introduce mass-action kinetic systems with time delays and show that they generate a nonnegative semiflow. 

Consider the ordinary mass-action kinetic system \cite{Feinberg1979}
\begin{equation}
\label{eq:mak_ode}
 \dot{x}(t)= \sum_{k=1}^{M} \kappa_{k}\, (x(t))^{y_k} \left[y_k' - y_k\right], \qquad t \geq 0,
\end{equation}
where $x(t) \in\overline{\mathbb{R}}^{N}_{+}$ is the state vector.  We have a set of complexes $\mathcal{K} \subset\overline{\mathbb{Z}}^{N}_{+}$ and there are $M$ reactions between the complexes. As usual,  $\overline{\mathbb{Z}}^{N}_{+}$ denotes the set on nonnegative integers. Each reaction has a source and product complex $y_k, y_k'\in \mathcal{K}$, respectively, with a reaction rate constant $\kappa_k > 0$, $k=1,\dots,M$. Solutions of
\eqref{eq:mak_ode} are determined by nonnegative initial vectors $x(0)=\eta\in\overline{\mathbb{R}}^{N}_{+}$.
In this paper, we will consider the mass-action kinetic system with time delays
\begin{equation}
\label{eq:delay_crn}
 \dot{x}(t) =\sum_{k=1}^{M} \kappa_{k}\, \left[(x(t-\tau_k))^{y_k}\,y_k' - (x(t))^{y_k}\,y_k\right], \qquad t \geq 0,
\end{equation}
where $\tau_k\geq0$, $k=1,\dots,M$. In the special case $\tau_k=0$, $k=1,\dots,M$, Eq.~\eqref{eq:delay_crn} reduces to the ordinary mass kinetic system~\eqref{eq:mak_ode}. Solutions of~\eqref{eq:delay_crn} are generated by initial data
$x(t)=\theta(t)$ for $-\tau\leq t\leq0$, where $\tau=\max_{1\leq k\leq M}\tau_k$ is the maximum delay and $\theta\in\overline{\mathcal{C}}_+$ is a nonnegative continuous initial function. Throughout the paper, the solution of \eqref{eq:delay_crn} with initial function~$\theta\in\overline{\mathcal{C}}_+$ will be denoted by $x=x^\theta$. Note that the solutions of delay differential equations are usually interpreted in~$\mathcal C$. For every $t\geq0$, $x_t\in\mathcal C$ is defined by $x_t(s)=x(t+s)$ for $-\tau\leq s\leq0$.

In the following theorem, we show that the semiflow generated by the time delay kinetic system \eqref{eq:delay_crn} is nonnegative.

\begin{thm}\label{thm:nonnegative}
For every initial function~$\theta\in\overline{\mathcal{C}}_+$, the solution $x^\theta$ of \eqref{eq:delay_crn} is nonnegative, i.e., $x^\theta_t\in\overline{\mathcal{C}}_{+}$ for all $t\geq0$.
\end{thm}
\begin{proof}
Eq.~\eqref{eq:delay_crn} can be written in the form
\begin{equation*}
\dot{x}(t)=F(x_t),
\end{equation*}
where $F:\overline{\mathcal{C}}_+\rightarrow\mathbb{R}^N$ is given by
\begin{equation*}
F(\phi)=
\sum_{k=1}^{M} \kappa_{k}\, \left[(\phi(-\tau_k))^{y_k}\,y_k' - (\phi(0))^{y_k}\,y_k\right],\qquad\phi\in\overline{\mathcal{C}}_{+}.
\end{equation*}
It follows from the definition of the vector exponential that if $\phi\in\overline{\mathcal{C}}_{+}$ and $\phi_i(0)=0$ for some~$i\in\{1,\dots,N\}$, then
\begin{equation*}
F_i(\phi)=\sum_{k=1}^{M} \kappa_{k}(\phi(-\tau_k))^{y_k}\,{(y_k')}_i\geq0.
\end{equation*}
Here $\phi_i$ and $F_i$ denote the $i$-th coordinate function of $\phi$ and~$F$, respectively. The conclusion follows from Theorem~2.1 in Chap.~5 of \cite{Smith1995}. Alternatively, we can use the generalization of Proposition~3.1 of \cite{Haddad2010} to equations with multiple delays.
\end{proof}

\section{Stoichiometric compatibility classes for delayed kinetic systems}

Recall \cite{Feinberg1979} that the stoichiometric subspace $\mathcal{S}$ for the ordinary mass-action kinetic system \eqref{eq:mak_ode} is defined by
\begin{equation}
\mathcal{S} = \text{span}\left\{y_k' - y_k ~\mid~ k = 1,\dots,M\right\},
\end{equation}
and for each $p \in\overline{\mathbb{R}}^{N}_{+}$ the corresponding positive stoichiometric compatibility class $\mathcal{S}_p$ is given by
\begin{equation}
\mathcal{S}_p = \left\{x \in\overline{\mathbb{R}}^{N}_{+} ~\mid~ x-p \in \mathcal{S} \right\}.
\end{equation}
 It is well know that the positive stoichiometric classes $\mathcal{S}_p$ are positively invariant under the mass-action kinetic system \eqref{eq:mak_ode}, i.e.  $x(0) \in \mathcal{S}_p$ implies $x(t)\in\mathcal{S}_p$ for all $t \geq 0$.

In this section we will extend the definition of the positive stoichiometric classes to the time delayed kinetic system~\eqref{eq:delay_crn} and we prove their invariance property.

For each $v\in{\mathbb{R}}^N$, define the functional
 $c_v:\overline{\mathcal{C}}_{+}\rightarrow\mathbb{R}$ by
\begin{equation}
\label{eq:funct_c}
c_v(\psi) = v^T \left[ \psi(0) + \sum_{k = 1}^{M} \left( \kappa_k \int_{-\tau_{k}}^0 (\psi(s))^{y_k}\, ds \right) y_k \right],\qquad\psi\in\overline{\mathcal{C}}_{+}.
\end{equation}
Let $\mathcal{S}^\perp$ denote the orthogonal complement of the stoichiometric subspace~$\mathcal{S}$ given by $\mathcal{S}^\perp=\{\,v\in\mathbb{R}^N\mid v^T y=0\text{ for all
 $y\in\mathcal{S}$}\,\}$. 
  Now we can formulate the definition of the positive stoichiometric compatibility classes for the delayed kinetic system~\eqref{eq:delay_crn}.
  For each $\theta\in\overline{\mathcal{C}}_{+}$, the {\it positive stoichiometric compatibility class of \eqref{eq:delay_crn} corresponding to~$\theta$} is denoted by $\mathcal{D}_\theta$ and is defined by
 \begin{equation}
\label{eq:delayed_stoch_comp_class}
\mathcal{D}_\theta =\{\psi\in\overline{\mathcal{C}}_{+} \mid c_v(\psi) = c_v(\theta)
\quad\text{for all $v \in \mathcal{S}^{\perp}$}\}.
\end{equation}
It is easily seen that  $\psi\in\mathcal{D}_\theta$ if and only if $\psi\in\overline{\mathcal{C}}_{+}$ and
\begin{equation}\label{eq:charact}
\psi(0) -\theta(0) + \sum_{k = 1}^{M} \left( \kappa_k \int_{-\tau_{k}}^0[ (\psi(s))^{y_k} - (\theta(s))^{y_k}]\, ds \right) y_k  \in \mathcal{S}.
\end{equation}
Therefore if we ignore the delays in \eqref{eq:delay_crn}, i.e. $\tau_k=0$ for $k=1,\dots,M$, then the above delayed positive stoichiometric compatibility classes coincide with the positive stoichiometric compatibility classes of the ordinary kinetic system~ \eqref{eq:mak_ode}.

In the next theorem, we establish the invariance property of the above delayed positive stoichiometric compatibility classes.

\begin{thm}\label{thm:invariance}
For every $\theta\in\overline{\mathcal{C}}_{+}$, the positive stoichiometric compatibility class~$\mathcal{D}_\theta$ is a closed subset of~$\overline{\mathcal{C}}_{+}$. Moreover,  $\mathcal{D}_\theta$ is positively invariant under Eq.~\eqref{eq:delay_crn}, i.e. if $\psi\in\mathcal{D}_\theta$, then $x^\psi_t\in\mathcal{D}_\theta$ for all $t\geq0$.
\end{thm}

\begin{proof}
Let $\theta\in\overline{\mathcal{C}}_{+}$. The closedness of~$\mathcal{D}_\theta$ is a simple 
consequence of the continuity of functionals~$c_v$, $v\in\mathcal{S}^\perp$.
We will show that for every $v\in\mathcal{S}^\perp$ the functional $c_v$ defined by \eqref{eq:funct_c} is constant along the solutions of Eq.~\eqref{eq:delay_crn}. Indeed, if $x$ is a solution of~\eqref{eq:delay_crn}, then we have for $t\geq0$,
\begin{equation*}
\begin{split}
\frac{d}{dt}({c}_v(x_t)) &= v^T \sum_{k=1}^{M} \kappa_{k}\, (x(t-\tau_k))^{y_k}\left(y_k' - y_k\right) \\
& = \sum_{k=1}^{M} \kappa_{k}\, (x(t-\tau_k))^{y_k}v^T(y_k' - y_k)=0,
\end{split}
\end{equation*}
the last equality being a consequence of the definition of~$\mathcal{S}^\perp$. From this, we find that if $\psi\in\mathcal{D}_\theta$, then for every $v\in\mathcal{S}^\perp$ and $t\geq0$,
$$
c_v(x^\psi_t)=c_v(x^\psi_0)=c_v(\psi)=c_v(\theta)
$$
and hence $x^\psi_t\in\mathcal{D}_\theta$. This show that $\mathcal{D}_\theta$
invariant under Eq.~\eqref{eq:delay_crn}. 
\end{proof}

\section{Semistability for delayed complex balanced kinetic systems}
Before we formulate our main stability criterion, we recall some definitions.

By a {\it positive equilibrium} of~\eqref{eq:mak_ode} or~\eqref{eq:delay_crn}, we mean a positive vector $\overline{x}\in\mathbb{R}^N_{+}$ such that $x(t)\equiv\overline{x}$ is a solution of~\eqref{eq:mak_ode} and~\eqref{eq:delay_crn}, respectively. Note that 
Eqs.~\eqref{eq:mak_ode} and~\eqref{eq:delay_crn} share the same equilibria satisfying the algebraic equation
\begin{equation}
\sum_{k=1}^{M} \kappa_{k}(\overline{x})^{y_k} \left[y_k' - y_k\right]=0.
\end{equation}
A positive equilibrium~$\overline{x}$ is called {\it complex balanced} if for every 
$\eta \in \mathcal{K}$,
\begin{equation}
\label{eq:complex_balanced}
\sum_{k:\eta = y_k}\kappa_k(\overline{x})^{y_k}= \sum_{k:\eta = y_k'} \kappa_k (\overline{x})^{y_k},
\end{equation}
where the sum on the left is over all reactions for which~$\eta$ is the source complex and the
sum on the right is over all reactions for which~$\eta$ is the product complex.
Finally, an ordinary or delayed kinetic system is called {\it complex balanced} if it has a positive complexed balanced equilibrium.

It is well-known \cite{vanderSchaft2015} that if Eq.~\eqref{eq:mak_ode} and hence~\eqref{eq:delay_crn}
 has a positive complex balanced equilibrium $\overline{x}$, then any other positive equilibrium is complex balanced and the set of all positive equilibria~$\mathcal{E}$ can be characterized by 
\begin{equation}\label{eq:equil}
\mathcal{E} = \{\tilde{x} \in\mathbb{R}^N_{+}\mid \Ln(\tilde{x}) - \Ln(\overline{x}) \in \mathcal{S}^\perp\}.
\end{equation}
Now we formulate the main result of the paper about the semistability of positive equilibria of delayed complex balanced systems in the sense of the following definition. A positive equilibrium $\overline{x}$ of Eq.~\eqref{eq:delay_crn} is called {\it semistable} if it is Lyapunov stable and there exists $\delta>0$ such that if $\theta\in\hat{\mathcal{B}}_\delta(\overline{x})$,
 then $x^\theta(t)$ converges to a Lyapunov stable equilibrium of~\eqref{eq:delay_crn} as $t\to\infty$. As usual, $\hat{\mathcal{B}}_\delta(\overline{x})=\{\,\psi\in\mathcal{C}\mid\|\psi-\overline{x}\|\leq\delta\,\}.
$
\begin{thm}\label{thm:semistability}
Every positive complex balanced equilibrium of the delayed kinetic system~\eqref{eq:delay_crn} is semistable.
\end{thm}
As a preparation for the proof of Theorem~\ref{thm:semistability}, we establish an auxiliary result about the uniqueness of  positive equilibria in the positive stoichiometric classes of complex balanced systems.

\begin{thm}\label{thm:unique}
Suppose that the delayed kinetic system~\eqref{eq:delay_crn} is complex balanced. 
Then for every $\theta\in\overline{\mathcal{C}}_{+}$ the corresponding delayed stoichiometric 
class~$\mathcal{D}_\theta$ contains at most one positive equilibrium.
\end{thm}

\begin{proof}Let $\theta\in\overline{\mathcal{C}}_{+}$. Suppose that $\tilde{x}$ and $\overline{x}$
are positive equilibria belonging to~$\mathcal{D}_\theta$. From the
characterization~\eqref{eq:charact} of~~$\mathcal{D}_\theta$, we find that
\begin{equation*}
\tilde{x} -\overline{x} + \sum_{k = 1}^{M} \left( \kappa_k \int_{-\tau_{k}}^0[\tilde{x}^{y_k} -\overline{x}^{y_k}]\, ds \right) y_k \in \mathcal{S}.
\end{equation*}
This, together with~\eqref{eq:equil}, yields
\begin{equation*}
\begin{split}
0 &= (\Ln(\tilde{x}) - \Ln(\overline{x}))^T\left[\tilde{x} -\overline{x} + \sum_{k = 1}^{M} \left( \kappa_k \int_{-\tau_{k}}^0[\tilde{x}^{y_k} - \overline{x}^{y_k}]\, ds \right) y_k \right] \\
&= \sum_{i=1}^{N}\left(\ln(\tilde{x}_i) - \ln(\overline{x}_i)\right)(\tilde{x}_i -\overline{x}_i) 
+ \sum_{k = 1}^{M} \kappa_k\tau_{k} \left(\ln(\tilde{x}^{y_k}) - \ln(\overline{x}^{y_k})\right)\left(\tilde{x}^{y_k} - \overline{x}^{y_k}\right).
\end{split}
\end{equation*}
Since $(\ln(a)-\ln(b))(a-b) \geq 0$ whenever $a$, $b>0$
with equality  if and only if $a = b$, this is possible only if $\tilde{x}_i=\overline{x}_i$
for all $i=1,\dots,N$.
\end{proof}

Now we are in a position to give a proof of Theorem~\ref{thm:semistability}. It will be based on the Lyapunov-Krasovskii method and LaSalle's invariance principle \cite{Haddad2010}, \cite{Hale1993}, \cite{Smith2011}, \cite{Anderson2014}.

\begin{proof}[Proof of Theorem~\ref{thm:semistability}]We will use the following two inequalities. For every $a,b\in\mathbb R$,
\begin{equation}
\label{eq:exp_ineq}
e^a(b-a) \leq e^b-e^a,
\end{equation}
with equality if and only if $a=b$.
For every $b>0$ there exists $c>0$ such that for all $x>0$,
\begin{equation}
\label{eq:log_ineq}
x\bigl[\ln(x)-\ln(b)- 1\bigr] + b \geq c \ln\bigl[\,1 + (x-b)^2\,\bigr]\geq0.
\end{equation}
Inequality~\eqref{eq:exp_ineq} is not new. It is a simple consequence of the mean value theorem applied to the exponential function. Inequality~\eqref{eq:log_ineq} is less obvious. Its proof is given in the Appendix.

Consider the candidate Lyapunov–Krasovskii functional  $V:\mathcal{C}_+ \rightarrow \overline{\mathbb{R}}_+$ defined by
\begin{equation}
\label{eq:lyap_kras_func}
\begin{split}
V(\psi) = &\sum_{i = 1}^{N}\bigl( \psi_i(0) (\ln(\psi_i(0)) - \ln(\overline{x}_i) - 1) + \overline{x}_i \bigr) \\
&+\sum_{k = 1}^{M} \kappa_k \int_{-\tau_{k}}^0 \left\{(\psi(s))^{y_k} \left[\ln((\psi(s))^{y_k}) - \ln(\overline{x}^{y_k}) - 1\right] + \overline{x}^{y_k}\right\} ds
\end{split}
\end{equation}
for $\psi\in\mathcal{C}_+$. Clearly, $V(\overline{x})=0$. We will show that there exists
a continuous strictly increasing function $\alpha:[0,\infty)\rightarrow[0,\infty)$ with  
$\alpha(0)=0$ such that
\begin{equation}\label{eq:lowerbound}
V(\psi)\geq\alpha(|\psi(0)-\overline{x}|),\qquad\psi\in\mathcal{C}_{+},
\end{equation}
where $|\cdot|$ is the Euclidean norm in~$\mathbb{R}^N$. 
By virtue of~\eqref{eq:log_ineq}, the second sum in~\eqref{eq:lyap_kras_func} is nonnegative and the first sum in~\eqref{eq:lyap_kras_func} can be estimated from below using the first inequality in~\eqref{eq:log_ineq}. Thus, \eqref{eq:log_ineq} implies the existence of
positive constants ~$c_i$, $1\leq i\leq N$, such that for all $\psi\in\mathcal{C}_+$,
\begin{equation*}
\begin{split}
V(\psi) \geq &\sum_{i = 1}^{N}c_i\ln\left[1 + (\psi_i(0) - \overline{x}_i)^2\right]
\geq\gamma\sum_{i = 1}^{N}\ln\left[1 + (\psi_i(0) - \overline{x}_i)^2\right]\\
&= \gamma \ln\prod_{i = 1}^{N}\left[1 + (\psi_i(0) - \overline{x}_i)^2\right]\geq \gamma\ln\biggl(1+\sum_{i=1}^N(\psi_i(0) - \overline{x}_i)^2\biggr)\\
& = \gamma \ln\left(1 + |\psi(0) - \overline{x}|^2\right),
\end{split}
\end{equation*}
where $\gamma=\min_{1\leq i\leq N}c_i$. Thus,   \eqref{eq:lowerbound} holds with
$$
\alpha(r)=\gamma\ln(1+r^2),\qquad r\geq0.
$$
Next, it follows that the Lyapunov-Krasovskii directional derivative along trajectories
 of~\eqref{eq:delay_crn} is given by
\begin{equation*}
\begin{split}
\dot{V}(x_t) &= \sum_{k = 1}^{M} \kappa_{k}\, \Ln\left(\frac{x(t)}{\overline{x}}\right)^T\left[(x(t-\tau_k))^{y_k}\,y_k' - (x(t))^{y_k}\,y_k\right] \\
&+\sum_{k = 1}^{M} \kappa_{k}\,(x(t))^{y_k}\, \left(\ln\left(\left\{\frac{x(t)}{\overline{x}}\right\}^{y_k}\right) - 1\right) \\
&-\sum_{k = 1}^{M} \kappa_{k} \, (x(t-\tau_k))^{y_k}\, \left(\ln\left(\left\{\frac{x(t-\tau_k)}{\overline{x}}\right\}^{y_k}\right) - 1\right)\\
&=  \sum_{k = 1}^{M} \kappa_{k}\,\left[(x(t-\tau_k))^{y_k}\,\ln\left(\left\{\frac{x(t)}{\overline{x}}\right\}^{y_k'}\right) - (x(t))^{y_k}\,\ln\left(\left\{\frac{x(t)}{\overline{x}}\right\}^{y_k}\right)\right] \\
&+\sum_{k = 1}^{M} \kappa_{k}\,\left[ (x(t))^{y_k}\, \ln\left(\left\{\frac{x(t)}{\overline{x}}\right\}^{y_k}\right)
- \, (x(t-\tau_k))^{y_k}\, \ln\left(\left\{\frac{x(t-\tau_k)}{\overline{x}}\right\}^{y_k}\right)\right] \\
&+ \sum_{k = 1}^{M} \kappa_{k}\left[(x(t-\tau_k))^{y_k} - (x(t))^{y_k}\right] \\
&=\sum_{k = 1}^{M} \kappa_{k}\,\overline{x}^{y_k} \left(\frac{x(t-\tau_k)}{\overline{x}}\right)^{y_k}\,\left[
\ln\left(\left\{\frac{x(t)}{\overline{x}}\right\}^{y_k'}\right) -  \ln\left(\left\{\frac{x(t-\tau_k)}{\overline{x}}\right\}^{y_k}\right)\right]  \\
&+ \sum_{k = 1}^{M} \kappa_{k}\,\overline{x}^{y_k}\,\left[\left(\frac{x(t-\tau_k)}{\overline{x}}\right)^{y_k} - \left(\frac{x(t)}{\overline{x}}\right)^{y_k}\right].
\end{split}
\end{equation*}
By virtue of~\eqref{eq:exp_ineq}, we have for each $k=1,\dots ,M$,
\begin{equation*}
\begin{split}
&\left(\frac{x(t-\tau_k)}{\overline{x}}\right)^{y_k}\left[
\ln\left(\left\{\frac{x(t)}{\overline{x}}\right\}^{y_k'}\right)-\ln\left(\left\{\frac{x(t-\tau_k)}{\overline{x}}\right\}^{y_k}\right)\right]\\
&\leq
\left(\frac{x(t)}{\overline{x}}\right)^{y_k'} -  \left(\frac{x(t-\tau_k)}{\overline{x}}\right)^{y_k}
\end{split}
\end{equation*}
with equality if and only if  for each $k=1,\dots, M$,
\begin{equation*}
\left(\frac{x(t)}{\overline{x}}\right)^{y_k'}=\left(\frac{x(t-\tau_k)}{\overline{x}}\right)^{y_k}.
\end{equation*}
From this, we find that
\begin{equation*}
\begin{split}
\dot{V}(x_t) &\leq\sum_{k = 1}^{M} \kappa_{k}\,\overline{x}^{y_k} \,\left[
\left(\frac{x(t)}{\overline{x}}\right)^{y_k'} - \left(\frac{x(t)}
{\overline{x}}\right)^{y_k}\right]\\
&= \sum_{\eta \in \mathcal{K}} \left(\frac{x(t)}{\overline{x}}\right)^{\eta}\left[ \sum_{k:\eta = y_k'}\kappa_k \overline{x}^{y_k} - \sum_{k:\eta = y_k} \kappa_k \overline{x}^{y_k} \right] = 0,
\end{split}
\end{equation*}
where the last equality follows from the complex balanced property~\eqref{eq:complex_balanced}.
This implies that the complex balanced equilibrium~$\overline{x}$ of \eqref{eq:delay_crn} is Lyapunov stable.

Choose $\epsilon$ such that $0<\epsilon<\min_{1\leq i\leq N}\overline{x}_i$ so that
$\hat{\mathcal{B}}_\epsilon(\overline{x})\subset\mathcal{C}_{+}$. The Lyapunov stability of the equilibrium~$\overline{x}$ implies the existence of $\delta>0$ such that if 
$\theta\in\hat{\mathcal{B}}_\delta(\overline{x})$, then $x^\theta_t\in\hat{\mathcal{B}}_\epsilon(\overline{x})$ for all $t\geq0$. We will show that for every 
$\theta\in\hat{\mathcal{B}}_\delta(\overline{x})$ the solution $x^\theta(t)$ converges to a Lyapunov stable equilibrium of~\eqref{eq:delay_crn}. Let 
$$
\mathcal{R}=\{\,\psi\in\hat{\mathcal{B}}_\epsilon(\overline{x})\mid\dot V(\psi)=0\,\}.
$$
From the previous calculations, we find that
\begin{equation*}
\mathcal{R} = \biggl\{\,\psi\in\hat{\mathcal{B}}_\epsilon(\overline{x})\mid \left(\frac{\psi(0)}{\overline{x}}\right)^{y_k'}=\left(\frac{\psi(-\tau_k)}{\overline{x}}\right)^{y_k}\quad\text{for $k=1,\dots,M$}\biggr\} .
\end{equation*}
Let $\mathcal{M}$ be the largest set in~$\mathcal{R}$ which is invariant under Eq.~\eqref{eq:delay_crn}. We will show that every element of~$\mathcal{M}$ is a positive equilibrium of~ \eqref{eq:delay_crn}. Let $\psi\in\mathcal{M}$ and write $x=x^\psi$ for brevity. Rewrite Eq.~\eqref{eq:delay_crn} in the form
\begin{equation*}
\dot{x}(t) = \sum_{\eta \in \mathcal{K}} \left[\sum_{k:\eta=y_k'} \kappa_k \overline{x}^{y_k}\left(\frac{x(t-\tau_k)}{\overline{x}}\right)^{y_k} - \sum_{k:\eta=y_k} \kappa_k \overline{x}^{y_k}\left(\frac{x(t)}{\overline{x}}\right)^{y_k} \right] \eta.
\end{equation*}
Since $\mathcal{M}\subset\mathcal{R}$ is invariant, we have that $x_t\in\mathcal{R}$ for all $t\geq0$ and hence
\begin{equation*}
\dot{x}(t)=\sum_{\eta \in \mathcal{K}} \left(\frac{x(t)}{\overline{x}}\right)^{y_k}\left[\sum_{k:\eta=y_k'} \kappa_k \overline{x}^{y_k} - \sum_{k:\eta=y_k} \kappa_k \overline{x}^{y_k} \right] \eta = 0
\end{equation*}
for $t\geq0$, where the last equality is a consequence of the complex balanced property \eqref{eq:complex_balanced}. Thus, $x=x^\psi$ is a constant solution of~\eqref{eq:complex_balanced} and hence $\psi\equiv\tilde x$ is a positive equilibrium. Now suppose that $\theta\in\hat{\mathcal{B}}_\delta(\overline{x})$. As noted before, $x^\theta_t\in\hat{\mathcal{B}}_\epsilon(\overline{x})$ for all $t\geq0$. By the application of LaSalle's invariance principle \cite{Smith2011}, we conclude that $\omega(\theta)\subset\mathcal{M}$, where 
$
\omega(\theta)=\{\,\phi\in\mathcal{C}\mid \text{there exists $t_n\rightarrow\infty$ such that 
$x^\theta_{t_n}\rightarrow\phi$}\,\}
$
is the omega limit set. On the other hand, since $\theta\in\mathcal{D}_\theta$ and according to Theorem~\ref{thm:invariance} the stoichiometric class $\mathcal{D}_\theta$ is closed and invariant, it follows that $\omega(\theta)\subset\mathcal{D}_\theta$. Thus,
 $\omega(\theta)\subset\mathcal{M}\cap\mathcal{D}_\theta$. As shown before, every element of~$\mathcal{M}$ is a positive equilibrium of~\eqref{eq:delay_crn}, while Theorem~\ref{thm:unique} implies that $\mathcal{D}_\theta$ contains at most one positive equilibrium. Hence $\omega(\theta)=\{\tilde x\}$ for some $\tilde x\in\mathcal{E}$ and $x^\theta(t)\rightarrow\tilde x$ as $t\to\infty$. The Lyapunov stability of the positive equilibrium~$\tilde x$ follows from the first part of the proof.
\end{proof}

\begin{remark}{\rm In the previous proof we have shown that for every positive initial function $\theta$ from a neighborhood of the positive complex balanced equilibrium~$\overline{x}$ of~\eqref{eq:delay_crn} the stoichiometric class~$\mathcal{D}_\theta$ contains 
exactly one positive equilibrium. A simple modification of the above proof can be used to show that if system \eqref{eq:delay_crn} is comlexed balanced then $\mathcal{D}_\theta$ contains exactly one positive equilibrium whenever the closure of the forward orbit $\mathcal{O}^+_{\theta}=\{\,x^\theta_t\mid t\geq0\,\}$  remains in~$\mathcal{C}_{+}$. This last condition certainly holds if the solution $x^\theta$ is persistent in the sense that 
$\liminf_{t\to\infty}x^\theta_i(t)>0$ for each $i=1,\dots,N$.}

\end{remark}

Let $\overline{x}$ be a positive complex balanced equilibrium of Eq.~\eqref{eq:delay_crn}.  Theorem~\ref{thm:unique} implies that $\overline{x}$ is the only positive equilibrium in its positive stoichiometric compatibility class~$\mathcal{D}_{\overline{x}}$. This, together with Theorem~\ref{thm:semistability} yields the following analogue of a known result for ordinary kinetic systems.

\begin{thm}\label{thm:relative_stability}
Every positive complex balanced equilibrium $\overline{x}$ of the delayed kinetic system \eqref{eq:delay_crn} is locally asymptotically stable relative to its positive stoichiometric compatibility class~$\mathcal{D}_{\overline{x}}$.
\end{thm}

\section{Example}
In this section, we will illustrate our results and notations on a simple example. The studied system is intentionally low dimensional in order to be able to simply illustrate the relations and differences between non-delayed and delayed kinetic systems.

Let the time delayed complex balanced kinetic system be given by a reversible
reaction $2X_1 \leftrightarrows X_2$ containing one undelayed and a delayed reaction as follows
\begin{equation*}
2X_1 \xrightarrow{\kappa_1=1} X_2,~~ X_2 \xrightarrow{\kappa_2=2, \tau_2} 2X_1.
\end{equation*}
Then, the corresponding time-delay differential equation is
\begin{equation}
\label{eq:exmp_1_dde}
\begin{split}
\dot{x}(t) &= 1\left((x_1(t))^2\left[\begin{array}{c}0\\1\end{array}\right] - (x_1(t))^2\left[\begin{array}{c}2\\0\end{array}\right]\right) \\
&+ 2 \left(x_2(t-\tau_2)\left[\begin{array}{c}2\\0\end{array}\right] - x_2(t)\left[\begin{array}{c}0\\1\end{array}\right]\right),
\end{split}
\end{equation}
where $x =[x_1,x_2]^T\in\overline{\mathbb{R}}^{2}_{+}$ are the states and $\tau_2$ is the time delay of the second reaction. It is easily verified that 
$[2, 2]^T$ is a positive complex balanced equilibrium of~ \eqref{eq:exmp_1_dde}.
The stoichiometric subspace is
\begin{equation*}
 \mathcal{S} = \text{span}\left\{\,[-2,1]^T\,\right\}\qquad\text{and}
 \qquad\mathcal{S}^\perp = \text{span}\left\{\,[1,2]^T\,\right\}.
\end{equation*}
The dimension of $S$ is one, therefore Eq.~\eqref{eq:exmp_1_dde} has infinitely many positive equilibria given by the set
\begin{equation}
\mathcal{E} = \left\{\overline{x} \in\mathbb{R}^{2}_{+} ~\mid~
\left[\begin{array}{r}
\ln(\overline{x}_1) - \ln(2) \\
\ln(\overline{x}_2) - \ln(2) \end{array} \right] \in \mathcal{S}^{\perp} \right\}.
\end{equation}
For $\overline{x}\in\mathcal{E}$, consider the set $\mathcal{X}_{\overline{x}}$ of those positive constant functions  which belong to~$\mathcal{D}_{\overline{x}}$:
\begin{equation}
\mathcal{X}_{\overline{x}} = \left\{\eta \in\mathbb{R}^{2}_{+} ~\mid~
\left[\begin{array}{r}\eta_1 - \overline{x}_1\\(1+ 2\tau_2)(\eta_2 - \overline{x}_2)\end{array}\right]\in \mathcal{S}
 \right\}.
\end{equation}
According to Theorem \ref{thm:relative_stability}, if $\theta\equiv\eta \in \mathcal{X}_{\overline{x}}$ is close to the equilibrium $\overline{x}\in\mathcal{E}$, then $x^\theta(t)\rightarrow \overline{x}$ as $t\to\infty$. 

Figure \ref{fig:exmp1_fig1} shows the phase portrait of the system \eqref{eq:exmp_1_dde} with $\tau_2 = 0.5$ and with different constant initial conditions. The initial conditions are chosen such that the corresponding solutions converge to three different equilibria.  Figure~\ref{fig:exmp1_fig2} shows the time domain behavior of the system \eqref{eq:exmp_1_dde} when there are different time delays, but the initial conditions are same.
\begin{figure}
\center
\includegraphics[scale=0.6]{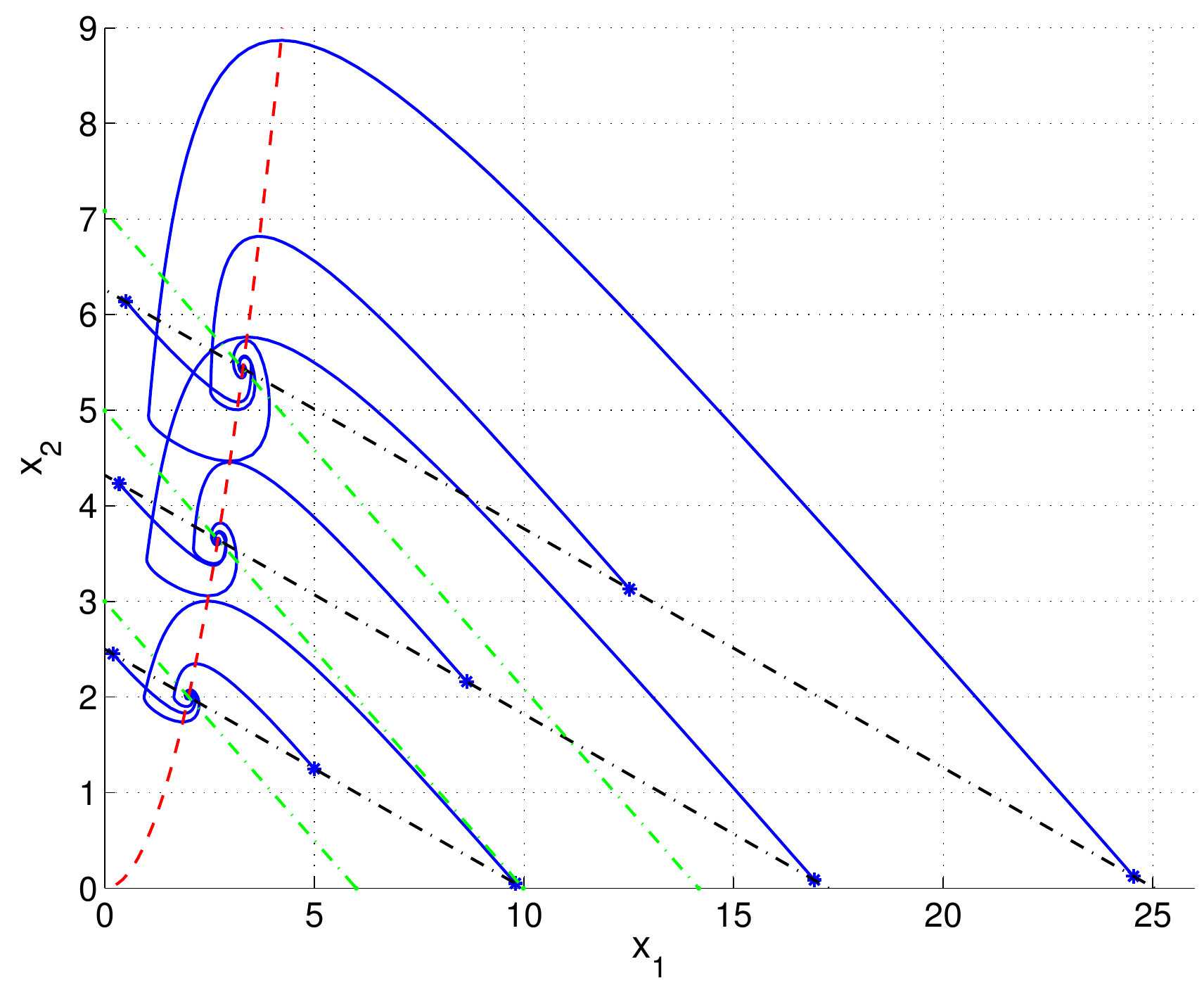}
\caption{
\label{fig:exmp1_fig1}The phase portrait of the system \eqref{eq:exmp_1_dde} with $\tau_2 = 0.5$. The red dash curve shows the equilibrium set $\mathcal{E}$ of the network. The black dash-dot lines show the set of points for which the corresponding constant initial functions result in the same equilibrium point. The green dashed lines show three stoichiometric compatibility classes of the non-delayed network having the same structure and reaction rate coefficients as the delayed one. The blue curves show the solution trajectories of \eqref{eq:exmp_1_dde} with different constant initial functions.}
\end{figure}

\begin{figure}
\center
\includegraphics[scale=0.6]{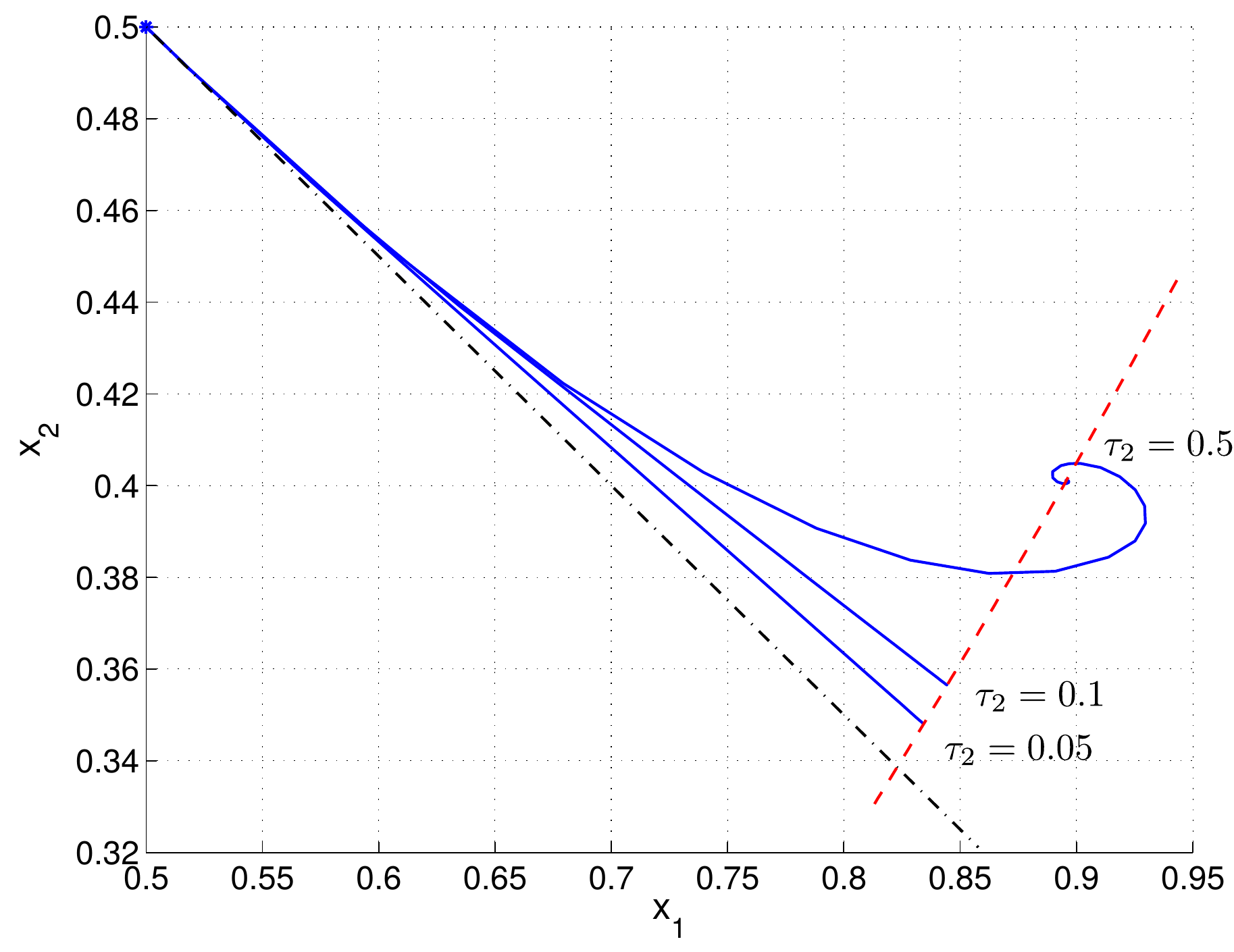}
\caption{\label{fig:exmp1_fig2} The phase portrait of the system \eqref{eq:exmp_1_dde} with different time delays $\tau_2 = \left\{0.05, 0.1, 0.5 \right\}$ and with the same constant initial function defined by $\eta = \left[ 0.5~ 0.5\right]^T$. The red dashed line shows the equilibrium set $\mathcal{E}$ of the network. The black dash-dot line shows the positive stoichiometric compatibility class of the undelayed system having the same structure and reaction rate coefficients as the delayed one. The blue curves show the solution trajectories of \eqref{eq:exmp_1_dde} with different time delays.}
\end{figure}

\section{Conclusions}
In this paper a class of delayed kinetic systems is introduced, where different constant time-delays can be assigned to the individual reactions of the network. The complex balance property is defined for this system in a straightforward way. It is shown that the equilibrium solutions of complex balanced kinetic systems can be directly obtained from the equilibria of the corresponding non-delayed kinetic system. Therefore, the complex balance property of a delayed network can be checked in the same way as in the non-delayed model. The notion of stoichiometric compatibility classes is extended to delayed networks. It is shown that contrary to the classical mass action case, these classes are no longer linear manifolds in the state space. The uniqueness of equilibrium solutions within a time delayed positive stoichiometric compatibility class is proved for delayed complex balanced models. By introducing a logarithmic Lyapunov-Krasovskii functional and using LaSalle's invariance principle, the semistability of equilibrium solutions in complex balanced systems with arbitrary time delays is also proved. As a consequence, a positive complex balanced equilibrium is always locally asymptotically stable relative to its positive stoichiometric class. The obtained results further underline the significance of the complex balance principle in the theory of dynamical systems. In the light of \cite{Anderson2011} and \cite{Craciun2015}, an interesting question is whether the asymptotic stability of a delayed complex balanced system is global relative to its positive stoichiometric class.

\section*{Acknowledgements}
We acknowledge the fruitful discussions with Professor Ferenc Hartung and Dr. Tamás Péni. This research has been supported by the National Research, Development and Innovation Office through grants K115694 and K120186. G. Sz. acknowledges the support of the grant PPKE KAP-1.1-16-ITK.


\section*{Appendix}
We give a proof of inequality~\eqref{eq:log_ineq}. Let $b>0$ be fixed. For $x>0$, define
\begin{equation*}
f(x)=x\left[\,\ln(x)-\ln(b)- 1\right] + b
\end{equation*}
and
\begin{equation*}
g(x)=\ln\bigl[1 + (x-b)^2\bigr].
\end{equation*}
Since $f'(x)=\ln(x)-\ln(b)$ for $x>0$, $f'<0$ on $(0,b)$ and $f'>0$ on $(b,\infty)$. This implies that~$f$ has a strict minimum at $x=b$. Hence
$f(x)>f(b)=0$ for $x\in(0,b)\cup(b,\infty)$. Clearly, the same inequality holds for~$g$. A repeated application of l'Hospital's rule yields
\begin{equation*}
\lim_{x\to b}\frac{f(x)}{g(x)}=\frac{1}{2b}.
\end{equation*}
Therefore the function $h:(0,\infty)\rightarrow\mathbb{R}$ defined by
\begin{equation*}
h(x)=\begin{cases}\dfrac{f(x)}{g(x)}&\qquad\text{for $x\in(0,b)\cup(b,\infty)$}\\
\dfrac{1}{2b}&\qquad\text{for $x=b$}
\end{cases}
\end{equation*}
is positive and continuous. Since $x\ln(x)\rightarrow0$ as $x\to0+$, $h$ can be extended continuously to the interval $[0,\infty)$ by
\begin{equation*}
h(0)=\lim_{x\to0+}h(x)=\frac{b}{\ln(1+b^2)}.
\end{equation*}
Since $\lim_{x\to\infty}h(x)=\infty$, there esists~$T>0$ such that $h(x)>h(0)$ for all $x> T$. The continuity of~$h$ implies the existence of
$c=\min_{0\leq x\leq T}h(x)$. Since $h(x)>h(0)\geq c>0$ for $x>T$, we have that $h(x)\geq c$ for all $x\geq0$ which implies the desired inequality~\eqref{eq:log_ineq}.

\end{document}